\documentclass[12pt,a4paper,reqno]{amsart}

\usepackage[margin=1in,footskip=0.25in]{geometry}
\usepackage{amsmath, amsthm, amssymb}
\usepackage{graphicx}
\newtheorem{teo}{Theorem}
\newtheorem{prop}{Proposition}

\newcommand{\re}{\mathbb{R}}

\date{}

\begin{document}

\title[Crossing limit cycles for a family of isochronous centers]
{Crossing limit cycles of planar discontinuous piecewise differential systems formed by isochronous centers}

\author{Claudio A. Buzzi$^1$}
\address{$^1$ Mathematics Department, Universidade Estadual Paulista Julio de Mesquita Filho, 15054-000 São José do Rio Preto, São Paulo, Brazil }
\email{claudio.buzzi@unesp.br}

\author{Yagor Romano Carvalho$^2$}
\address{$^2$ Mathematics Department, Universidade de São Paulo, 13566-590 São Carlos, São Paulo, Brazil}
\email{yagor.carvalho@usp.br}

\author{Jaume Llibre$^3$}
\address{$^3$ Mathematics Department, Universitat Aut\`{o}noma de Barcelona, 08193 Bellaterra, Barcelona, Catalonia, Spain}
\email{jllibre@mat.uab.cat}

\subjclass[2010]{37G15, 37D45.}

\keywords{Limit cycles, linear centers, cubic isochronous centers with homogeneous nonlinearities, discontinuous piecewise differential systems, first integrals}

\begin{abstract}
These last years an increasing interest appeared for studying the planar discontinuous piecewise differential systems motivated by the rich applications in modelling real phenomena. One of the difficulties for understanding the dynamics of these systems is the study their limit cycles. In this paper we study the maximum number of crossing limit cycles of some classes of planar discontinuous piecewise differential systems separated by a straight line, and formed by combinations of linear centers (consequently isochronous) and cubic isochronous centers with homogeneous nonlinearities. For these classes of planar discontinuous piecewise differential systems we solved the extension of the 16th Hilbert problem, i.e. we provide an upper bound for their maximum number of crossing limit cycles.
\end{abstract}

\maketitle

\section{Introduction and statement of the main results}

A {\it limit cycle} is a periodic orbit of a differential system in $\re^2$ that is isolated in the set of all its periodic orbits.  The analysis of the existence of limit cycles became important in the applications of the real world because many phenomena are related to their existence, see for instance the Van der Pol oscillator \cite{Van1920,Van1926}. The study of limit cycles began with Poincaré \cite{Poincare1897} at the end of the nineteenth century.  On the other hand, the study of the continuous piecewise linear differential systems separated by a straight line has special attention from the mathematicians, mainly because these systems appear in a natural way in the control theory, see for instance the books \cite{HenMicDale1997,LliTer2014,Narendra2014,Ogata1990}, in mechanics, electrical circuits, economy, see for instance the books \cite{BerBudCham2008,SimJohn2010} and the surveys \cite{MakLam2012,Teixeira2012}.  

The easiest continuous piecewise linear differential systems are formed by two linear differential systems separated by a straight line and it is known that such systems have at most one limit cycle, see \cite{FrePonRodTor1998,LliOrdPon2013,LumChu1991,LumChu1992}.  But it is also known that if both linear differential systems are linear centers, then the continuous piecewise linear differential system has no limit cycles, see for example \cite{LliTei2018}. However if we eliminate the continuity of such systems, that is, they do not need to coincide on the line of discontinuity, then it is known that these systems can have three limit cycles as we can see in \cite{BuzPesTor2013,DenLuis2013, FreiPonTor2014, HunYan2012, Lip2014, LliPon2012}, but it is unknown if three is the maximum number of limit cycles that they can have.

When we consider planar discontinuous piecewise differential systems we can have two kinds of limit cycles: {\it sliding limit cycles} or {\it crossing limit cycles}. The first ones contain some segment of the line of discontinuity, and the second ones only contain some points of the line of discontinuity. For more details on the discontinuous piecewise differential systems see the books \cite{BerBudCham2008, Fil1988, MakLam2012, SimJohn2010}.  In this work we are going to study the crossing limit cycles, and in what follows sometimes when we talk about limit cycles, we are talking about crossing limit cycles.  

An {\it isochronous center} of a planar differential system is a singularity such that there is a neighborhood of it formed by periodic orbits with the same period. It is well known that the linear centers are isochronous.
 
Llibre and Teixeira in \cite{LliTei2018} were interested in studying if a discontinuous piecewise differential system formed with only linear centers can create limit cycles. In this work we are interested in a similar problem. The first objective is to study the maximum number of limit cycles that discontinuous piecewise differential systems separated by a straight line can have when one differential system is a linear center and the other is a cubic isochronous center with homogeneous nonlinearities. The second objective is to study the maximum number of limit cycles that a discontinuous piecewise differential systems separated by a straight line can have when both differential systems are cubic isochronous centers with homogeneous nonlinearities. Without loss of generality we can assume that the straight line of discontinuity is $x=0$.

Consider the polynomial differential systems of the form
\begin{equation}\label{sistemacubico}
\begin{array}{l}
\dfrac{dx}{dt}=\dot x = -y + a_{30}x^3 +a_{21}x^2 y+ a_{12}xy^2 +a_{03}y^3, \vspace{0.2cm} \\
\dfrac{dy}{dt}=\dot y = x + b_{30}x^3 +b_{21}x^2 y+ b_{12}xy^2 +b_{03}y^3.
\end{array}
\end{equation}
Pleshkan in \cite{Pleshkan1969} classify which of these differential systems have an isochronous center at the origin of coordinates. Thus a cubic system \eqref{sistemacubico} has an isochronous center at the origin if and only if the system can be transformed to one of the following four differential systems
\begin{equation}\label{sistemas1s2s3s4}
\begin{aligned}
&(S_1^3): \; \left. \begin{array}{l}
\dot{x}=-y + x^3 - x y^2, \\
\dot{y}=x + x^2y - y^3,
\end{array} \right. \; \; \; \; \; \; \; 
(S_2^3): \; \left. \begin{array}{l}
\dot{x}=-y + x^3 - 3 x y^2, \\
\dot{y}=x + 3x^2y - y^3, 
\end{array} \right.  \\
& (S_3^3): \; \left. \begin{array}{l}
\dot{x}=-y + 3x^2y,  \\
\dot{y}=x -2x^3 +9 x y^2,
\end{array} \right. \; \; \; \; \; \;
(S_4^3): \; \left. \begin{array}{l}
\dot{x}=-y - 3x^2y, \\
\dot{y}=x + 2x^3-9xy^2.
\end{array} \right.
\end{aligned}
\end{equation}
doing a linear change of coordinates and a rescaling of time.

The first integrals of the differential systems \eqref{sistemas1s2s3s4} can be founded in \cite{ChaSab1999}, and they are
\begin{equation*}\label{intprimeiras1s2s3s4}
\begin{aligned}
&(S_1^3): H_1(x,y)=\dfrac{x^2+y^2}{1+2xy}, \; \; \; \; \; \; \; \; \; \; \; \; \; \; \; \;  \; \; \; \;  \; \; \; (S_2^3): H_2(x,y)=\dfrac{(x^2+y^2)^2} {1+4xy},\\
&(S_3^3): H_3(x,y)=\dfrac{x^2+y^2-4x^4+4x^6}{(-1+3x^2)^3}, \; \; \; \; \; (S_4^3): H_4(x,y)=\dfrac{x^2+y^2+4x^4+4x^6}{(1+3x^2)^3},
\end{aligned}
\end{equation*}
respectively.   

Llibre and Teixeira proved in \cite{LliTei2018} that after doing an affine transformation and a rescaling of the independent variable any linear center can be written into the form
\begin{equation}\label{sistemalinear}
(L_c): \; \begin{array}{l}
\dot{x}= - A x -\dfrac{4 A^2 + \omega^2}{4 D}  y + B, \vspace{0.2cm} \\
\dot{y}=D x + A y + C,
\end{array}
\end{equation}
where $A, \; B, \; C, \; D, \; \omega$ are real numbers with  $D, \; \omega>0$. This system has first integral
\begin{equation}\label{intprimeiralinear}
H_L(x,y) = 4 ( D x + A y)^2 + 8 D (C x - B y) + y^2 \omega^2.
\end{equation}

In this paper we only study discontinuous piecewise differential systems separated by a straight line and formed by two differential systems which can be transformed, after an affine transformation, in differential systems belonging to some of the three classes  $(S_1^3)$, $(S_2^3)$ and $(L_c)$. Here we do not take into account a rescaling of the independent variable because such a rescaling does not change the phase portrait of the differential systems, only change the speed in which their orbits are travelled.

Let $A,B \in \{(L_c),(S_1^3),(S_2^3)\}$ be. Then we denote by $N_{clc}[A,B]$ the maximum number of crossing limit cycles that the class of the discontinuous piecewise differential systems separated by a straight line and formed by the differential systems $A$ and $B$ can have.   

Our results are summarized in the next theorem.

\begin{teo}\label{teofinalcap4}
Consider the class of discontinuous piecewise  differential systems separeted by a straight line and formed by two differential systems which after an affine change of variables belong to the classes $(L_c)$, $(S_1^3)$ or $(S_2^3)$. Then
\begin{itemize}
\item [(i)] $N_{clc}[(L_c),(S_1^3)]=N_{clc}[(S_1^3),(S_1^3)]=1$;
		
\item [(ii)] $N_{clc}[(L_c),(S_2^3)]=2$;
		
\item [(iii)] $1 \leq N_{clc}[(S_1^3),(S_2^3)] \leq 3$;
		
\item [(iv)] $1 \leq N_{clc}[(S_2^3),(S_2^3)] \leq 9$;
		
\item[(v)] $N_{clc}[(L_c),(L_c)]=0$.
\end{itemize}
\end{teo}

In the following table we summarize the results of Theorem \ref{teofinalcap4}, and where we have added the numbers between parentheses, which mean the maximum number of known limit cycles for the corresponding class of discontinuous piecewise differential systems. Inside the proof of Theorem \ref{teofinalcap4} we will provide the corresponding discontinuous piecewise differential systems realizing the number of limit cycles which appear between parentheses.

\[
\begin{tabular}{|c|c|c|c|}
\hline
$\;$        & Linear & $(S_1^3)$   & $(S_2^3)$ \\
\hline
Linear     & $0$        & $1 \; (1) $ & $2 \; (2)$ \\
\hline
$(S_1^3)$  & $1 \; (1) $ & $1 \; (1) $ & $3 \; (1)$ \\
\hline
$(S_2^3)$ &  $2 \; (2)$ &  $3 \; (1)$ & $9 \; (1)$ \\
\hline
\end{tabular} 
\]

\vspace{0.4cm}

We want to stress that our contribution in Theorem \ref{teofinalcap4} restricts to items (i), (ii), (iii) and (iv). As we have already said item (v) was proved, for example in \cite{LliTei2018}.

\section{Definitions and Preliminaries} 

In this section we review some definitions and previous results that will be used for proving our results. Let 
\begin{equation}\label{sistemanoplano}
\dot{x}=P(x,y), \qquad \dot{y}=Q(x,y),
\end{equation}
be a planar polynomial differential system. Then a nonconstant analytic function $H=H(x,y)$ is a {\it first integral} of system \eqref{sistemanoplano} if it is constant on all solution curves $(x(t),y(t))$, that is, $H(x(t),y(t))=constant$ for all values of $t$ for which $H(x(t),y(t))$ is defined.  Clearly $H$ is a first integral of system \eqref{sistemanoplano} if and only if 
$$
XH = P\frac{\partial H}{\partial x} + Q \frac{\partial H}{\partial y} = 0.
$$ 
Therefore the solutions of system \eqref{sistemanoplano} are contained in the level curves of the function $H$.  A differential system in the plane with a first integral is said {\it integrable}. For more details on integrable systems see Chapter 8 of \cite{DumLliArt2006}.

Consider a planar discontinuous piecewise differential system separated by the straight line $x=0$. A periodic orbit of a such system must intersect the line $x = 0$ exactly in two points. Since we always will work with integrable systems, let $H^1$ and $H^2$ be the two first integrals of the two differential systems forming the discontinuous piecewise differential system. If we have a limit cycle which intersect $x=0$ in the two points by $(0,y_1)$ and $(0,y_2)$ with $y_1< y_2$, then
\begin{equation}\label{eqciclos}
H^1(0,y_1)-H^1 (0,y_2) = 0, \quad \mbox {and} \quad H^2(0,y_1) -H^2(0,y_2) = 0.
\end{equation}
Therefore our objective will be to control how many solutions the system \eqref{eqciclos} has, and to try to find discontinuous piecewise differential systems satisfy this number of solutions. 

It is important to note that every solution $y_1<y_2$ of system \eqref{eqciclos} in general does not provide a limit cycle. For instance when the system has a continuum of solutions; or when the level curves of either $H^1$ or $H^2$ through the points $(0,y_1)$ and $(0,y_2)$ are disconnected; or when the two pieces of the level curves of $H^1$ and $H^2$ through the points $(0,y_1)$ and $(0,y_2)$ provide a closed curve, but the two pieces of the two orbits of the two differential systems forming the discontinuous piecewise differential system have different orientation. In summary, every crossing limit cycle provides a unique solution of system \eqref{eqciclos} with $y_1<y_2$, but a solution of the system \eqref{eqciclos} does not necessarily provides a crossing limit cycle. 

Since the system \eqref{eqciclos} for our discontinuous piecewise differential systems always can be reduced to a polynomial system, a good tool that we will use for estimating its number of solutions will be the Bezout Theorem. This theorem says that if a polynomial system has a finite number of solutions, then the number of its solutions is at most the product of the degrees of the polynomials that appear in the system, for more details see \cite{Shaf1974, Vainsencher1996}. 

We also must take into account that if we have a polynomial system 
$$ 
F(y_1,y_2) = 0,  \qquad G(y_1,y_2) = 0, 
$$ 
satisfying that $F(y_1,y_2) = F(y_2,y_1)$ and $G(y_1,y_2) = G(y_2,y_1)$, then $(y_1,y_2)$ and $(y_2,y_1)$ are solutions. Since we only are interested in the solutions $y_1<y_2$. So the number of solutions of system \eqref{eqciclos} must be divided by two, in order to obtain an upper bound for the number of limit cycles.

Moreover, whenever we have a symmetric polynomial system (invariant by permutations of its variables), as our system \eqref{eqciclos} sometimes it will be convenient to do the change of variables $(y_1,y_2) \to (z,w)$ given by $z= y_1+y_2$ and $w=y_1y_2$, in order to study its solutions.  For more details see Section III.4 of \cite{GarLeq2006}.

Doing the general affine change of variables $x = aX + bY + c$, $y = \alpha X + \beta Y + \gamma $, with $ a \beta - b \alpha \neq 0 $, the first integral of the system $(S_1^3)$ is transformed into
\begin{equation*}\label{sistemaiso1mud}
H_{1c}(X,Y)= \dfrac{(c + a X + b Y)^2 + (X \alpha + Y \beta + \gamma)^2}{1 + 2 (c + a X + b Y) + (X \alpha + Y \beta + \gamma)},
\end{equation*}
and the first integral of the system $(S_2^3)$ is transformed into
\begin{equation*}\label{sistemaiso2mud}
H_{2c}(X,Y)= \dfrac{((c + a X + b Y)^2 + (X \alpha + Y \beta + \gamma)^2)^2}{1 + 4 (c + a X + b Y)  (X \alpha + Y \beta + \gamma)}.
\end{equation*}

The normal form of a general linear center and its first integral are given in \eqref{sistemalinear}  and \eqref{intprimeiralinear}, respectively.

\section{Proof of Theorem $\ref{teofinalcap4}$}

\begin{figure}[!htb]
\begin{center}
\includegraphics[scale=0.4]{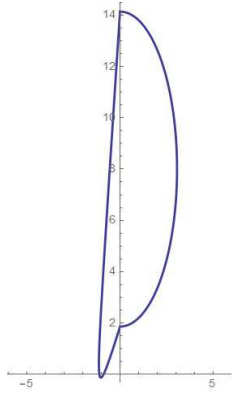}
\caption{The limit cycle of the discontinuous piecewise differential system \eqref{sistema1positvo}-\eqref{sistema1negativo}  of Proposition \ref{teo1}.}\label{cicloteo1}
\end{center}
\end{figure}

The proof that $N_{clc}[(L_c),(S_1^3)]=1$ is given in the next proposition.

\begin{prop}\label{teo1}
Consider the class of discontinuous piecewise differential systems separated by the straight line $x=0$ and formed by a linear center and a cubic isochronous center $(S_1^3)$ after an arbitrary affine change of variables.  Then these differential systems can have at most one limit cycle.  Moreover, the discontinuous piecewise differential system in this class formed by the differential system
\begin{equation}\label{sistema1positvo}
\dot{x}=2 - \dfrac{y}{4}, \qquad \dot{y}=x,
\end{equation}
in $x\ge 0$; and by the differential system
\begin{equation}\label{sistema1negativo}
\begin{array}{l}
\dot{x}=  \dfrac{1}{2}(23 + 38 x + 7 y + 16 x^2 + 16 x y - 3 y^2 + 8 x^2 y - 2 x y^2), \vspace{0.2cm} \\
\dot{y}=12 + 12 x + 13 y + 16 x y + 4 x y^2 - y^3, 
\end{array}
\end{equation}
in $x\le 0$,
has one limit cycle, reaching the maximum upper bound. See Figure 1.
\end{prop}

\begin{proof}
Under the assumptions of the proposition if such discontinuous piecewise differential systems have a limit cycle intersecting the discontinuity straight line $x=0$ in the two points $(0,y_1)$ and $(0,y_2)$, we have that $y_1$ and $y_2$ must satisfy that
\begin{equation*}
H_{1c}(0, y_1) - H_{1c}(0, y_2) = 0, \quad  \mbox{and} \quad
H_L(0, y_1) - H_L(0, y_2) = 0,
\end{equation*}
or equivalently
\begin{equation*}
\begin{aligned}
&(y_1 - y_2)( 8 B D - 4 A^2 y_1 - 4 A^2 y_2 - y_1 \omega^2 - y_2 \omega^2)=0, \\			
&(y_1 - y_2)(-2 b c - b^2 y_1 - b^2 y_2 + 2 c^3 \beta + 2 b c^2 y_1 \beta + 2 b c^2 y_2 \beta + 2 b^2 c y_1 y_2 \beta - y_1 \beta^2  \\
&-y_2 \beta^2 - 2 c y_1 y_2 \beta^3 - 2 b c^2 \gamma - 2 b^2 c y_1 \gamma - 2 b^2 c y_2 \gamma - 2 b^3 y_1 y_2 \gamma - 2 \beta \gamma - 2 c y_1 \beta^2 \gamma \\
&- 2 c y_2 \beta^2 \gamma + 2 b y_1 y_2 \beta^2 \gamma - 2 c \beta \gamma^2 + 2 b y_1 \beta \gamma^2 + 2 b y_2 \beta \gamma^2 + 2 b \gamma^3)/ ((1 + 2 c y_1 \beta\\
&+ 2 b y_1^2 \beta + 2 c \gamma +  2 b y_1 \gamma) (1 + 2 c y_2 \beta + 2 b y_2^2 \beta + 2 c \gamma + 2 b y_2 \gamma))=0 .
\end{aligned}
\end{equation*}
Since $(1 + 2 c y_1 \beta+ 2 b y_1^2 \beta + 2 c \gamma + 2 b y_1 \gamma) (1 + 2 c y_2 \beta + 2 b y_2^2 \beta + 2 c \gamma + 2 b y_2 \gamma) \neq 0 $ and we have $y_1 \neq y_2$, then we obtain the equivalent system
\begin{equation}\label{sistemageral1}
\begin{aligned}
& F_1^1(y_1,y_2)= 8 B D - 4 A^2 y_1 - 4 A^2 y_2 - y_1 \omega^2 - y_2 \omega^2=0 ,\\
& F_2^1(y_1,y_2)=  -2 b c - b^2 y_1 - b^2 y_2 + 2 c^3 \beta + 2 b c^2 y_1 \beta + 2 b c^2 y_2 \beta + 2 b^2 c y_1 y_2 \beta - y_1 \beta^2 \\
& \phantom{F_2^1(y_1,y_2)=} -y_2 \beta^2 - 2 c y_1 y_2 \beta^3 - 2 b c^2 \gamma - 2 b^2 c y_1 \gamma - 2 b^2 c y_2 \gamma - 2 b^3 y_1 y_2 \gamma - 2 \beta \gamma \\
& \phantom{F_2^1(y_1,y_2)=}  - 2 c y_1 \beta^2 \gamma - 2 c y_2 \beta^2 \gamma + 2 b y_1 y_2 \beta^2 \gamma -  2 c \beta \gamma^2 + 2 b y_1 \beta \gamma^2 +  2 b y_2 \beta \gamma^2 \\
& \phantom{F_2^1(y_1,y_2)=}  + 2 b \gamma^3=0.
\end{aligned}
\end{equation}
So in system \eqref{sistemageral1} we have that $F_1^1(y_1,y_2)$ is a polynomial of degree $1$ and $F_2^1(y_1,y_2)$ is a polynomial of degree $2$. Eventually system \eqref{sistemageral1} could have a continuum of solutions $(y_1,y_2)$, but then the possible periodic solutions would not be limit cycles. Therefore we assume that this system has finitely many solutions. Then by Bezout Theorem this system has at most two solutions. Since if $(y_1^*,y_2^*)$ is a solution, also $(y_2^*,y_1^*)$ is a solution, but we are interested in the solutions whose first component be smaller than the second one, so it follows that the discontinuous piecewise differential system has at most one limit cycle.
	
Now we shall prove that the discontinuous piecewise differential system separated by the straight line $x=0$ and defined by the differential systems (\ref{sistema1positvo}) and (\ref{sistema1negativo}) has one limit cycle. For $x\ge 0$ the system (\ref{sistema1positvo}) is a linear center and for $x\le 0$ the system is the cubic isochronous center $(S_1^3)$ after the affine change of variables $(x,y)\to (3 + 2 x, 1 + 2 x - y)$. The first integrals of this piecewise differential system are
\begin{equation*} 
H_{L1}(x,y) = 4 x^2 - 16 y + y^2 \; \; \; \mbox{and} \; \; \;  H_{1c_1}(x,y)= \dfrac{(3 + 2 x)^2 + (1 + 2 x - y)^2}{1 + 2 (3 + 2 x) (1 + 2 x - y)}.
\end{equation*}
Then the system (\ref{sistemageral1}) for this piecewise differential system is
\[ 
\begin{aligned}
& -16 + y_1 + y_2=0,\\
& \dfrac{-(-46 - 7 y_1 - 7 y_2 + 6 y_1 y_2)}{(-7 + 6 y_1) (-7 + 6 y_2))} =0, 
\end{aligned} 
\]
with the solution $(y_1^{1*},y_2^{1*})= \frac{1}{3} \left( 24 - \sqrt{339},  24 + \sqrt{339} \right)$, and  observe that we have $y_1^{1*}<y_2^{1*}$.  

The solution $(x_1^1(t), y_1^1(t))$ in $x\ge 0$ of system (\ref{sistema1positvo}) such that $(x_1^1(0),y_1^1(0)) = (0, y_1^{1*})$ is contained in the level curve $= H_{L1}(0,y_1^{1*})= - \frac{79}{3} = H_{L1}(0,y_2^{1*})$, i.e. in the curve
\[ 
\begin{aligned}
H_{L1}(x,y) = 4 x^2 - 16 y + y^2 = - \dfrac{79}{3}.
\end{aligned} 
\]
The solution $(x_2^1(t), y_2^1(t))$ in $x\le 0$ of system (\ref{sistema1negativo}) such that $(x_2^1(0),y_2^1(0)) = (0, y_2^{1*})$ is contained in the level curve $H_{1c_1}(0,y_1^{1*})= - \dfrac{7}{3} = H_{1c_1}(0,y_2^{1*})$, i.e. in the curve
\[ 
\begin{aligned}
H_{1c_1}(x,y)= \dfrac{(3 + 2 x)^2 + (1 + 2 x - y)^2}{1 + 2 (3 + 2 x) (1 + 2 x - y)} = - \dfrac{7}{3}.
\end{aligned} 
\]

Drawing the orbits $(x_k^1 (t), y_k^1 (t))$, $k=1,2$,  we obtain the limit cycle of Figure \ref{cicloteo1}.
\end{proof}

The proof that $N_{clc}[(L_c),(S_2^3)]=2$ is given in the next proposition.

\begin{figure}[!htb]
\begin{center}
\includegraphics[scale=0.4]{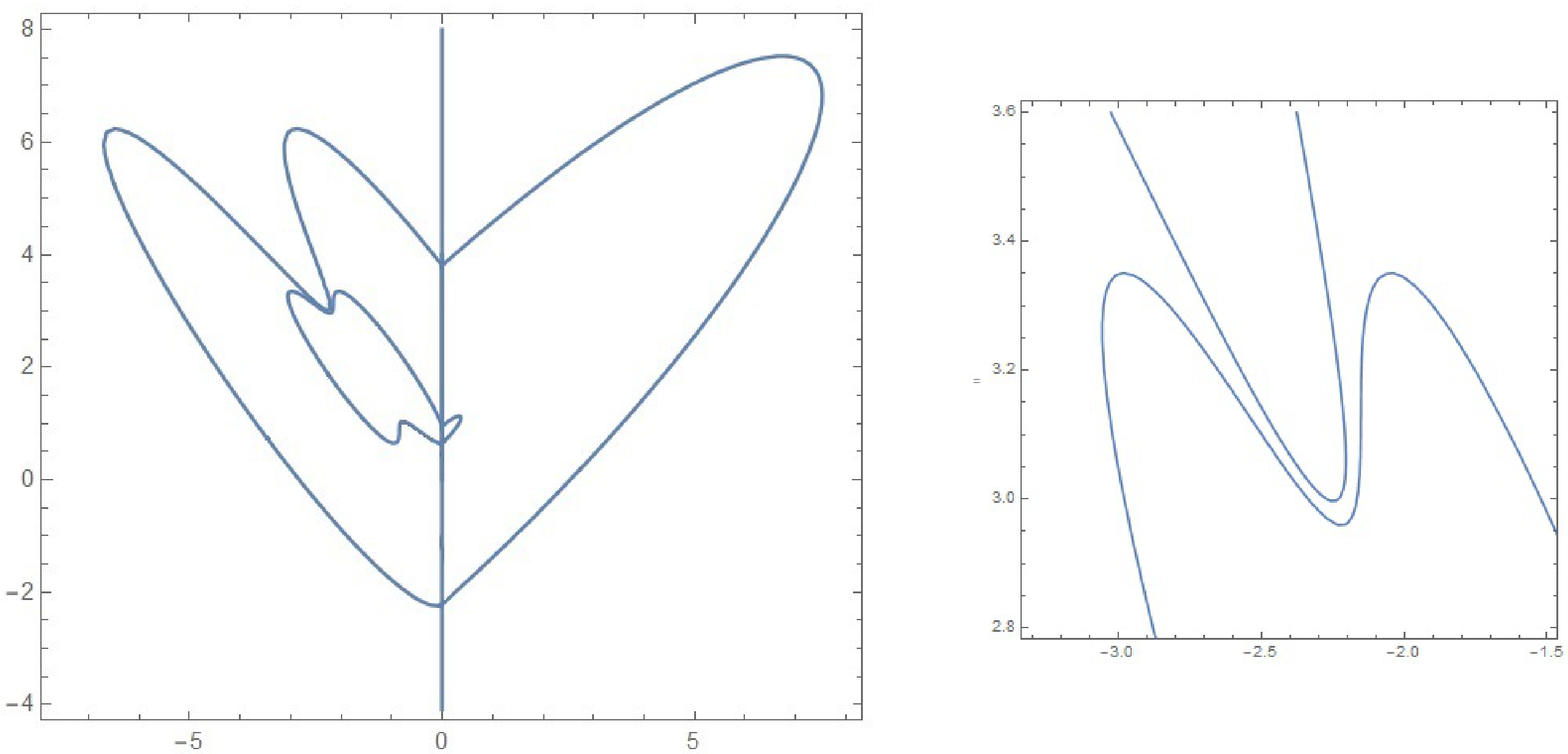}
\caption{The two limit cycles of the discontinuous piecewise differential system \eqref{sistema2positivo}-\eqref{sistema2negativo} of Proposition \ref{teo2}. The figure on the right shows that there is no intersection between the two limit cycles.}\label{cicloteo2}
\end{center}
\end{figure}

\begin{prop}\label{teo2}
Consider the class of discontinuous piecewise differential systems separated by the straight line $x=0$ and formed by a linear center and a cubic isochronous center $(S_2^3)$ after an arbitrary affine change of variables. Then these differential systems can have at most two limit cycles.  Moreover, the discontinuous piecewise differential system in this class formed by the differential system
\begin{equation}\label{sistema2positivo}
\dot{x}=1 + x - \dfrac{5}{4}y, \qquad 		\dot{y}=\dfrac{4}{5}+x-y.
\end{equation}
in $x\ge 0$, and the differential system
\begin{equation}\label{sistema2negativo}
\begin{array}{l}
\dot{x}= \dfrac{1}{2}(-1 + 30 x + 25 y + 24 x^2 - 15 y^2 - 48 x^3 - 120 x^2 y - 90 x y^2 - 20 y^3), \vspace{0.2cm} \\
\dot{y}= -20 x - 15 y + 24 x y + 18 y^2 + 32 x^3 + 72 x^2 y + 48 x y^2 + 9 y^3,
\end{array} 
\end{equation}
in $x\le 0$, has two limit cycles, reaching the maximum uppeer bound. See Figure 2.
\end{prop}

\begin{proof}
Under the assumptions of the proposition if such discontinuous piecewise differential systems have a limit cycle intersecting the discontinuity straight line $x=0$ in the two points $(0,y_1)$ and $(0,y_2)$, then $y_1$ and $y_2$ must satisfy that
\begin{equation}\label{x1}
H_{2c}(0, y_1) - H_{2c}(0, y_2) = 0, \quad  \mbox{and} \quad H_L(0, y_1) - H_L(0, y_2) = 0,
\end{equation}
or equivalently
\begin{equation*}\label{sistemageral2}
\begin{aligned}
&G_1^2(y_1,y_2)=- (y_1 - y_2)( 8 B D - 4 A^2 y_1 - 4 A^2 y_2 - y_1 \omega^2 - y_2 \omega^2)=0,\\
&G_2^2(y_1,y_2)=-((y_1 - y_2)P_2^2(y_1,y_2)/((1 + 4 c y_1 \beta + 4 b y_1^2 \beta + 4 c \gamma + 4 b y_1 \gamma) \\
&\phantom{G_2^2(y_1,y_2)=}(1 + 4 c y_2 \beta + 4 b y_2^2 \beta + 4 c \gamma + 4 b y_2 \gamma)))=0,
\end{aligned}
\end{equation*}
where $P_2^2(y_1,y_2)$ is a polynomial of degree $5$. Since $(1 + 4 c y_1 \beta + 4 b y_1^2 \beta + 4 c \gamma + 4 b y_1 \gamma) (1 + 4 c y_2 \beta + 4 b y_2^2 \beta + 4 c \gamma + 4 b y_2 \gamma) \neq 0$ and $y_1< y_2$, we can remove these terms to solve the system, and we get the equivalent system
\begin{equation}\label{sistemafinalcaso2}
\begin{aligned}
&F_1^2(y_1,y_2)=8 B D - 4 A^2 y_1 - 4 A^2 y_2 - y_1 \omega^2 - y_2 \omega^2=0, \\
&F_2^2(y_1,y_2)=P_2^2(y_1,y_2)=0.
\end{aligned}
\end{equation}
Eventually system \eqref{sistemafinalcaso2} could have a continuum of solutions $(y_1,y_2)$, but then the possible periodic solutions would not be limit cycles. Therefore we assume that this system has finitely many solutions. From $F_1^2(y_1,y_2)=0$ we obtain that
\begin{equation}\label{valory1}
y_1= \dfrac{8 B D - 4 A^2 Y - Y \omega^2}{4 A^2 + \omega^2},
\end{equation}
with $4 A^2 + \omega^2 \neq 0$ because $\omega>0$. So if we substitute \eqref{valory1} in $F_2^2(y_1,y_2)=0$, then we get a  polynomial $p_2(y_2)$ of degree $4$ in the variable $y_2$, and $p_2(y_2)$ has at most four roots. Therefore the system \eqref{sistemafinalcaso2} has at most four solutions, and consequently the discontinuous piecewise differential system can have at most two limit cycles.

Now we shall prove that the discontinuous piecewise differential system separated by the straight line $x=0$ and formed by the linear center (\ref{sistema2positivo}) and the cubic isochronous center (\ref{sistema2negativo}) obtained from $(S_2^3)$ after the affine change of variables $(x,y)\to (-1 - 2 x - y,1 - 2 x - 2 y)$,  has two limit cycles. 

The first integrals of this piecewise differential system are
$$ 
H_{L2}(x,y) = 4(x-y)^2 + 8 \left( \dfrac{4}{5}x - y \right) + y^2 , \;   H_{2c_1}(x,y)= \dfrac{((-1-2x-y)^2 + (1-2x-2y)^2)^2}{1+ 4 (-1-2x-y) (1-2x-2y)}.
$$
So system \eqref{x1} is equivalent to
\[ 
\begin{aligned}
&-8 + 5 y_1 + 5 y_2=0, \\
& 8 - 104 y_1 + 60 y_1^2 - 75 y_1^3 - 104 y_2 + 220 y_1 y_2 - 155 y_1^2 y_2 + 100 y_1^3 y_2 + 60 y_2^2 - \\
&-155 y_1 y_2^2 - 60 y_1^2 y_2^2 + 200 y_1^3 y_2^2 - 75 y_2^3 + 100 y_1 y_2^3 + 200 y_1^2 y_2^3=0	,
\end{aligned} 
\]
with the solutions
\[ 
\begin{aligned}
&(y_{11}^{2*},y_{12}^{2*})= \left( \dfrac{1}{10} \left( 8 - \sqrt{2(227 - 5 \sqrt{2041})}\right), \dfrac{1}{10} \left( 8 + \sqrt{2(227 - 5 \sqrt{2041})}\right) \right), \\ 
&(y_{21}^{2*},y_{22}^{2*})= \left( \dfrac{1}{10} \left( 8 - \sqrt{2(227 + 5 \sqrt{2041})}\right), \dfrac{1}{10} \left( 8 + \sqrt{2(227 + 5 \sqrt{2041})}\right) \right).
\end{aligned} 
\]
Observe that we have $y_{11}^{2*}<y_{12}^{2*}$ and $y_{21}^{2*}<y_{22}^{2*}$.  

The solution $(x_{11}^2(t), y_{11}^2(t))$ in $x\ge 0$ of system (\ref{sistema2positivo}) such that $(x_{11}^2(0),y_{11}^2(0)) = (0, y_{11}^{2*})$ is contained in the level curve $H_{L2}(0,y_{12}^{2*}) =  \dfrac{1}{2} \left( 39 - \sqrt{2041 }\right) = H_{L2}(0,y_{11}^{2*})$, i.e. in the curve
\[ 
\begin{aligned}
H_{L2}(x,y) = 4(x-y)^2 +8 \left(\frac{4}{5}x -y \right) +y^2= \dfrac{1}{2} \left( 39 - \sqrt{2041 }\right).
\end{aligned} 
\]
The solution $(x_{21}^2(t), y_{21}^2(t))$ in $x\le 0$ of system (\ref{sistema2negativo}) such that $(x_{21}^2(0),y_{21}^2(0)) = (0, y_{12}^{2*})$ is contained in the the level curve $H_{2c_1}(0,y_{12}^{2*})=  \dfrac{1}{14} \left( 263 - 5 \sqrt{2041 }\right) = H_{2c_1}(0,y_{11}^{2*})$, i.e. in the curve 
\[ 
\begin{aligned}
H_{2c_1}(x,y)=  \dfrac{((-1-2x-y)^2 + (1-2x-2y)^2)^2}{1+ 4 (-1-2x-y) (1-2x-2y)}  = \dfrac{1}{14} \left( 263 - 5 \sqrt{2041 }\right).
\end{aligned} 
\]

The solution $(x_{12}^2(t), y_{12}^2(t))$ in $x\ge 0$ of system (\ref{sistema2positivo}) such that $(x_{11}^2(0),y_{11}^2(0)) = (0, y_{21}^{2*})$ is contained in the level $H_{L2}(0,Y_{22}^{2*}) =  \dfrac{1}{2} \left( 39 + \sqrt{2041 }\right) = H_{L2}(0,y_{21}^{2*})$, i.e. in the curve
\[ 
\begin{aligned}
H_{L2}(x,y) = 4(x-y)^2 +8 \left(\frac{4}{5}x -y \right) +y^2= \dfrac{1}{2} \left( 39 + \sqrt{2041 }\right).
\end{aligned} 
\]
The solution $(x_{22}^2(t), y_{22}^2(t))$ in $x\le 0$ of system (\ref{sistema2negativo}) such that $(x_{22}^2(0),y_{22}^2(0)) = (0, Y_{22}^{2*})$ is contained in the level curve $H_{2c_1}(0,Y_{22}^{2*})=  \dfrac{1}{14} \left( 263 + 5 \sqrt{2041 }\right) = H_{2c_1}(0,y_{21}^{2*})$, i.e. in the curve
\[
\begin{aligned}
H_{2c_1}(x,y)= \dfrac{((-1-2x-y)^2 + (1-2x-2y)^2)^2}{1+ 4 (-1-2x-y) (1-2x-2y)}  = \dfrac{1}{14} \left( 263 + 5 \sqrt{2041 }\right).
\end{aligned} 
\]

Drawing the orbits $(x_{kj}^2 (t), y_{kj}^2 (t))$, $k,j=1,2$, we obtain the two limit cyles of Figure \ref{cicloteo2}.
\end{proof}

The proof that $N_{clc}[(S_1^3),(S_1^3)]=1$ is given in the next proposition.

\begin{figure}[!htb]
\begin{center}
\includegraphics[scale=0.3]{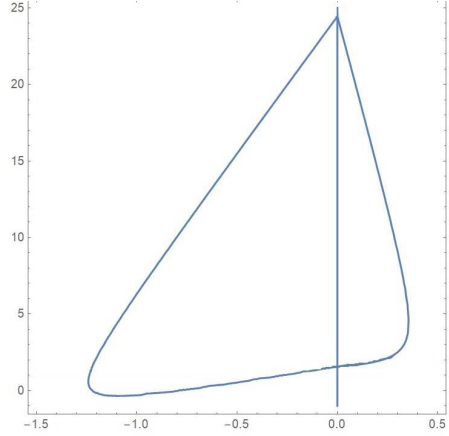}
\caption{The limit cycle of the discontinuous piecewise differential system \eqref{sistema3positivo}-\eqref{sistema3negativo} of Proposition \ref{teo3}.}\label{cicloteo3}
\end{center}
\end{figure}

\begin{prop}\label{teo3}
Consider the class of discontinuous piecewise differential systems separated by the straight line $x=0$ and formed by two distinct cubic isochronous center $(S_1^3)$ after an arbitrary affine change of variables. Then these differential systems can have at most one limit cycles.  Moreover, the discontinuous piecewise differential system in this class formed by the differential system	
\begin{equation}\label{sistema3positivo}
\begin{array}{l}
\dot{x}= - \dfrac{1}{2}(-1 - 4 x + 3 y + 12 x^2 - 4 x y - y^2 - 8 x^3 + 2 x y^2), \vspace{0.2cm} \\
\dot{y}= 1 - 6 x + 2 y + 4 x^2 + 4 x y - 3 y^2 - 4 x^2 y + y^3,
\end{array} 
\end{equation}
in $x\ge 0$	and by the differential system
\begin{equation}\label{sistema3negativo}
\begin{array}{l}
\dot{x}=  \dfrac{1}{2}(23 + 38 x + 7 y + 16 x^2 + 16 x y - 3 y^2 + 8 x^2 y - 2 x y^2), \vspace{0.2cm} \\
\dot{y}= 12 + 12 x + 13 y + 16 x y + 4 x y^2 - y^3, 
\end{array} 
\end{equation}
in $x\le 0$ has one limit cycle, reaching the maximum upper bound. See Figure 3.
\end{prop}

\begin{proof} 
Under the assumptions of the proposition if such discontinuous piecewise differential systems have a limit cycle intersecting the discontinuity straight line $x=0$ in the two points $(0,y_1)$ and $(0,y_2)$, then $y_1$ and $y_2$ must satisfy that 
\begin{equation}\label{x2}
H_{3c}(0, y_1) - H_{3c}(0, y_2) = 0, \quad \mbox{and} \quad H_{3_1c}(0, y_1) - H_{3_1c}(0, y_2) = 0.  
\end{equation}
where
\begin{equation*}
H_{3c}(X,Y)= \dfrac{(c + a X + b Y)^2 + (X \alpha + Y \beta + \gamma)^2}{1 + 2 (c + a X + b Y) + (X \alpha + Y \beta + \gamma)},
\end{equation*}
and
\begin{equation*}
H_{3_1c}(X,Y)= \dfrac{(c_1 + a_1 X + b_1 Y)^2 + (X \alpha_1 + Y \beta_1 + \gamma_1)^2}{1 + 2 (c_1 + a_1 X + b_1 Y) + (X \alpha_1 + Y \beta_1 + \gamma_1)}.
\end{equation*}

System \eqref{x2} is
\begin{equation*}
\begin{aligned}
&G_1^3(y_1,y_2)=-(y_1 - y_2) (-2 b c - b^2 y_1 - b^2 y_2 + 2 c^3 \beta + 2 b c^2 y_1 \beta + 2 b c^2 y_2 \beta   \\
&\phantom{G_1^3(y_1,y_2)=}+ 2 b^2 c y_1 y_2 \beta -y_1 \beta^2 - y_2 \beta^2 - 2 c y_1 y_2 \beta^3 - 2 b c^2 \gamma - 2 b^2 c y_1 \gamma - 2 b^2 c y_2 \gamma \\
&\phantom{G_1^3(y_1,y_2)=}-2 b^3 y_1 y_2 \gamma - 2 \beta \gamma - 2 c y_1 \beta^2 \gamma - 2 c y_2 \beta^2 \gamma + 2 b y_1 y_2 \beta^2 \gamma - 2 c \beta \gamma^2 \\ 
&\phantom{G_1^3(y_1,y_2)=}+2 b y_1 \beta \gamma^2 + 2 b y_2 \beta \gamma^2 + 2 b \gamma^3)/((1 + 2 c y_1 \beta + 2 b y_1^2 \beta + 2 c \gamma + 2 b y_1 \gamma) \\
&\phantom{G_1^3(y_1,y_2)=} \; (1 + 2 c y_2 \beta + 2 b y_2^2 \beta + 2 c \gamma + 2 b y_2 \gamma))=0, \\
&G_2^3(y_1,y_2)=-(y_1 - y_2) (-2 b_1 c_1 - b_1^2 y_1 - b_1^2 y_2 + 2 c_1^3 \beta_1 +	2 b_1 c_1^2 y_1 \beta_1 + 2 b_1 c_1^2 y_2 \beta_1  \\
&\phantom{G_2^3(y_1,y_2)=}+2 b_1^2 c_1 y_1 y_2 \beta_1 - y_1 \beta_1^2 - y_2 \beta_1^2 - 2 c_1 y_1 y_2 \beta_1^3 - 2 b_1 c_1^2 \gamma_1 - 2 b_1^2 c_1 y_1 \gamma_1 \\
&\phantom{G_2^3(y_1,y_2)=}- 2 b_1^2 c_1 y_2 \gamma_1 - 2 b_1^3 y_1 y_2 \gamma_1 - 2 \beta_1 \gamma_1 - 2 c_1 y_1 \beta_1^2 \gamma_1 - 2 c_1 y_2 \beta_1^2 \gamma_1 \\
&\phantom{G_2^3(y_1,y_2)=}+2 b_1 y_1 y_2 \beta_1^2 \gamma_1 - 2 c_1 \beta_1 \gamma_1^2 + 2 b_1 y_1 \beta_1 \gamma_1^2 + 2 b_1 y_2 \beta_1 \gamma_1^2 + 2 b_1 \gamma_1^3))/\\
&\phantom{G_2^3(y_1,y_2)=}((1 + 2 c_1 y_1 \beta_1 + 2 b_1 y_1^2 \beta_1 + 2 c_1 \gamma_1 + 2 b_1 y_1 \gamma_1)\\
&\phantom{G_2^3(y_1,y_2)=}(1 + 2 c_1 y_2 \beta_1 + 2 b_1 y_2^2 \beta_1 + 2 c_1 \gamma_1 + 2 b_1 y_2 \gamma_1))=0,
\end{aligned}
\end{equation*}
Since the denominators in the previous system cannot be zero and $y_1< y_2$, we obtain the equivalent system
\begin{equation}\label{sistemageral3}
\begin{aligned}
&F_1^3(y_1,y_2)= -2 b c - b^2 y_1 - b^2 y_2 + 2 c^3 \beta + 2 b c^2 y_1 \beta + 2 b c^2 y_2 \beta + 2 b^2 c y_1 y_2\beta  \\
&\phantom{G_1^3(y_1,y_2)=}  -y_1 \beta^2 - y_2 \beta^2 - 2 c y_1 y_2 \beta^3 - 2 b c^2 \gamma - 2 b^2 c y_1 \gamma - 2 b^2 c y_2 \gamma  \\
&\phantom{G_1^3(y_1,y_2)=}-2 b^3 y_1 y_2 \gamma - 2 \beta \gamma - 2 c y_1 \beta^2 \gamma - 2 c y_2 \beta^2 \gamma + 2 b y_1 y_2 \beta^2 \gamma - 2 c \beta \gamma^2 \\ 
&\phantom{G_1^3(y_1,y_2)=}+2 b y_1 \beta \gamma^2 + 2 b y_2 \beta \gamma^2 + 2 b \gamma^3=0, \\
&F_2^3(y_1,y_2)= -2 b_1 c_1 - b_1^2 y_1 - b_1^2 y_2 + 2 c_1^3 \beta_1 +	2 b_1 c_1^2 y_1 \beta_1 + 2 b_1 c_1^2 y_2 \beta_1  \\
&\phantom{G_2^3(y_1,y_2)=}+2 b_1^2 c_1 y_1 y_2 \beta_1 - y_1 \beta_1^2 - y_2 \beta_1^2 - 2 c_1 y_1 y_2 \beta_1^3 - 2 b_1 c_1^2 \gamma_1 - 2 b_1^2 c_1 y_1 \gamma_1 \\
&\phantom{G_2^3(y_1,y_2)=}- 2 b_1^2 c_1 y_2 \gamma_1 - 2 b_1^3 y_1 y_2 \gamma_1 - 2 \beta_1 \gamma_1 - 2 c_1 y_1 \beta_1^2 \gamma_1 - 2 c_1 y_2 \beta_1^2 \gamma_1 \\
&\phantom{G_2^3(y_1,y_2)=}+2 b_1 y_1 y_2 \beta_1^2 \gamma_1 - 2 c_1 \beta_1 \gamma_1^2 + 2 b_1 y_1 \beta_1 \gamma_1^2 + 2 b_1 y_2 \beta_1 \gamma_1^2 + 2 b_1 \gamma_1^3=0.
\end{aligned}
\end{equation}

The system \eqref{sistemageral3} could have a continuum of solutions $(y_1,y_2)$, but then the possible periodic solutions would not be limit cycles.  Therefore we assume that this system  has finitely many solutions.

Assume that $q_{31}(y_1)\ne 0$. From the equation $F_1^4(y_1,y_2)=0$ we obtain
\begin{equation}\label{valory2}
\begin{aligned}
& y_2= \dfrac{p_{31}(y_1)}{q_{31}(y_1)},
\end{aligned}
\end{equation}	
where 
\[ 
\begin{aligned}
& p_{31}(y_1)= -2 b c + 2 c^3 \beta - 2 b c^2 \gamma - 2 \beta \gamma - 2 c \beta \gamma^2 + 2 b \gamma^3 + \\
& \phantom{p_{31}(y_1)=}+y_1 (-b^2 + 2 b c^2 \beta - \beta^2 - 2 b^2 c \gamma - 2 c \beta^2 \gamma + 2 b \beta \gamma^2) , \\
& q_{31}(y_1)= b^2 - 2 b c^2 \beta + \beta^2 + 2 b^2 c \gamma + 2 c \beta^2 \gamma - 2 b \beta \gamma^2 + \\
& \phantom{q_{31}(y_1)=} + y_1 (-2 b^2 c \beta + 2 c \beta^3 + 2 b^3 \gamma - 2 b \beta^2 \gamma).
\end{aligned} 
\]
We substitute \eqref{valory2} in $F_2^3(y_1,y_2)=0$, and we get a rational function. The polynomial $p_3(y_1)$ in the numerator has degree $2$. Therefore $p_3(y_1)$ has at most two real roots, and consequently system \eqref{sistemageral3} has at most two solutions. Hence the discontinuous piecewise differential system has at most one limit cycle.

Now assume that $q_{31}(y_1)= 0$. Then also $p_{31}(y_1)= 0$, and then at most one solution for system \eqref{sistemageral3}. 
	
In summary, the discontinuous piecewise differential systems of this propositon have at most one limit cycle. Now we shall prove that the discontinuous piecewise differential system formed by the differential systems (\ref{sistema3positivo}) and (\ref{sistema3negativo}) has one limit cycle. 

In $x\ge 0$ the differential system (\ref{sistema3positivo}) comes from the cubic isochronous center $(S_1^3)$ after doing the affine change of variables $(x,y)\to (1-y,1-2 x)$, and in $x\le 0$ the differential system (\ref{sistema3negativo}) also comes from the cubic isochronous center $(S_1^3)$ after doing the affine change of variables $(x,y)\to (3+2x,1+2 x-y)$. The first integrals of these two differential system  are
\begin{equation*}		
H_{3c_1}(x,y) = \dfrac{(1-y)^2+ (1-2x)^2}{1 + 2  (1-y)(1-2x)}, \;  \mbox{and} \;   H_{3_1c_1}(x,y)= \dfrac{(3 + 2 x)^2 + (1 + 2 x - y)^2}{1 + 2 (3 + 2 x) (1 + 2 x - y)},
\end{equation*}
respectively. Then system (\ref{sistemageral3}) for this piecewise differential system becomes
\[ 
\begin{aligned}
& 2 - 3 y_1 - 3 y_2 + 2 y_1 y_2=0,\\
& -46 - 7 y_1 - 7 y_2 + 6 y_1 y_2 =0.
\end{aligned} 
\]
Its solution is $(y_1^{3*},y_2^{3*})= \left(13 - \sqrt{131}, 13 + \sqrt{131}\right)$, observe that we have $y_1^{3*}<y_2^{3*}$.  

The solution $(x_1^3(t), y_1^3(t))$ in $x\ge 0$ of system (\ref{sistema3positivo}) such that $(x_1^3(0),y_1^3(0)) = (0, y_1^{3*})$ is contained in the level curve $H_{3c_1}(0,y_2^{3*})= -12 = H_{3c_1}(0,y_1^{3*})$, i.e. in the curve
\[ 
\begin{aligned}
H_{3c_1}(x,y) = \dfrac{(1-y)^2+ (1-2x)^2}{1 + 2  (1-y)(1-2x)} = -12.
\end{aligned} 
\]
The solution $(x_2^3(t), y_2^3(t))$ in $x\le 0$ of system (\ref{sistema3negativo}) such that $(x_2^3(0),y_2^3(0)) = (0,y_2^{3*})$ is contained in the level curve $H_{3_1c_1}(0,y_2^{3*})= -4 = H_{3_1c_1}(0,y_1^{3*})$, i.e. in the curve
\[ 
\begin{aligned}
H_{3_1c_1}(x,y)= \dfrac{(3 + 2 x)^2 + (1 + 2 x - y)^2}{1 + 2 (3 + 2 x) (1 + 2 x - y)}  = - 4.
\end{aligned} 
\]

Drawing the orbits $(x_k^3 (t), y_k^3 (t))$, $k=1,2$, we obtain the limit cycle of Figure \ref{cicloteo3}.
\end{proof}

\begin{figure}[!htb]
\begin{center}
\includegraphics[scale=0.4]{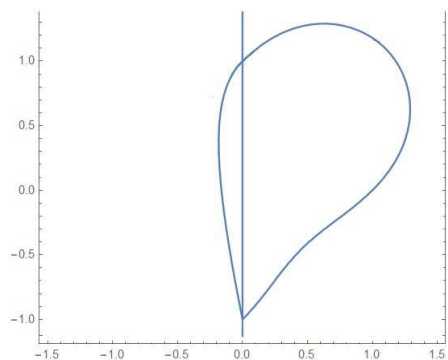}
\caption{The limit cycle of the discontinuous piecewise differential system \eqref{sistema4positivo}-\eqref{sistema4negativo} of Proposition \ref{teo4}.}\label{cicloteo4}
\end{center}
\end{figure}

The proof that $N_{clc}[(S_1^3),(S_2^3)]\le 3$ is given in the next proposition.

\begin{prop}\label{teo4}
Consider the class of discontinuous piecewise differential systems separated by the straight line $x=0$ and formed by two cubic isochronous centers of type $(S_1^3)$ and $(S_2^3)$ after an arbitrary affine change of variables. Then these differential systems can have at most three limit cycles.  Moreover, the discontinuous piecewise differential system in this class formed by the differential system		
\begin{equation}\label{sistema4positivo}
\begin{array}{l}
\dot{x}= -1 - x + 3 y + 3 x^2 + 4 x y - y^2 + x^3 - x y^2, \\
\dot{y}= 7 - 3 x - 11 y - 2 x^2 + 2 x y + 6 y^2 + x^2 y - y^3,
\end{array}, 
\end{equation}
in $x\ge 0$, and by the differential system
\begin{equation}\label{sistema4negativo}
\dot{x}=  y-x^3  + 3 x y^2, \qquad \dot{y}= -x - 3 x^2 y + y^3,
\end{equation}
in $x\le 0$, has one limit cycle.
\end{prop}

\begin{proof}  
Under the assumptions of the proposition if such discontinuous piecewise differential systems have a limit cycle intersecting the discontinuity straight line $x=0$ in the two points $(0,y_1)$ and $(0,y_2)$, then $y_1$ and $y_2$ must satisfy that 
\begin{equation}\label{x3}
H_{1c}(0,y_1)- H_{1c}(0,y_2)= 0, \quad \mbox{and} \quad H_{2_1c}(0, y_1) - H_{4_2c}(0, y_2) = 0
\end{equation}
where
\begin{equation*}
H_{1c}(x,y)= \dfrac{(c + a x + b y)^2 + (x \alpha + y \beta + \gamma)^2}{1 + 2 (c + a x + b y) (x \alpha + y \beta + \gamma)},
\end{equation*}
and
\begin{equation*}
H_{2_1c}(x,y)= \dfrac{((c_1 + a_1 X + b_1 Y)^2 + (X \alpha_1 + Y \beta_1 + \gamma_1)^2)^2}{1 + 4 (c_1 + a_1 X + b_1 Y) (X \alpha_1 + Y \beta_1 + \gamma_1)}.
\end{equation*}
Then system \eqref{x3} becomes
\begin{equation*}
\begin{aligned}
&G_1^4(y_1,y_2)=-(y_1 - y_2) (-2 b c - b^2 y_1 - b^2 y_2 + 2 c^3 \beta + 2 b c^2 y_1 \beta + 2 b c^2 y_2 \beta  \\
&\phantom{G_1^4(y_1,y_2)=}+2 b^2 c y_1 y_2 \beta - y_1 \beta^2 - y_2 \beta^2 - 2 c y_1 y_2 \beta^3 - 2 b c^2 \gamma - 2 b^2 c y_1 \gamma - 2 b^2 c y_2 \gamma \\
&\phantom{G_1^4(y_1,y_2)=}- 2 b^3 y_1 y_2 \gamma - 2 \beta \gamma - 2 c y_1 \beta^2 \gamma - 2 c y_2 \beta^2 \gamma + 2 b y_1 y_2 \beta^2 \gamma - 2 c \beta \gamma^2 \\ 
&\phantom{G_1^4(y_1,y_2)=}+2 b y_1 \beta \gamma^2 + 2 b y_2 \beta \gamma^2 + 2 b \gamma^3)/((1 + 2 c y_1 \beta + 2 b y_1^2 \beta + 2 c \gamma + 2 b y_1 \gamma) \\
&\phantom{G_1^4(y_1,y_2)=} \;(1 + 2 c y_2 \beta + 2 b y_2^2 \beta + 2 c \gamma + 2 b y_2 \gamma)))=0, \\
&G_2^4(y_1,y_2)=-(((y_1 - y_2)(G_{22}^4(y_1,y_2))))/\\
&\phantom{G_2^4(y_1,y_2)=} \; ((1 + 4 c_1 y_1 \beta_1 + 4 b_1 y_1^2 \beta_1 + 4 c_1 \gamma_1 + 4 b_1 y_1 \gamma_1) \\
& \phantom{G_2^4(y_1,y_2)=} \; (1 + 4 c_1 y_2 \beta_1 + 4 b_1 y_2^2 \beta_1 + 4 c_1 \gamma_1 + 4 b_1 y_2 \gamma_1))=0,
\end{aligned}
\end{equation*}
where $G_{22}^4(y_1,y_2)$ ís a  polynomial of degree $5$.
Since the denominators in the previous system cannot be zero and $y_1<y_2$, the previous system reduces to the system
\begin{equation}\label{sistemageral4}
\begin{aligned}
&F_1^4(y_1,y_2)= -2 b c - b^2 y_1 - b^2 y_2 + 2 c^3 \beta + 2 b c^2 y_1 \beta + 2 b c^2 y_2 \beta +2 b^2 c y_1 y_2 \beta \\
&\phantom{G_1^4(y_1,y_2)=}- y_1 \beta^2 - y_2 \beta^2 - 2 c y_1 y_2 \beta^3 - 2 b c^2 \gamma - 2 b^2 c y_1 \gamma - 2 b^2 c y_2 \gamma \\
&\phantom{G_1^4(y_1,y_2)=}- 2 b^3 y_1 y_2 \gamma - 2 \beta \gamma - 2 c y_1 \beta^2 \gamma - 2 c y_2 \beta^2 \gamma + 2 b y_1 y_2 \beta^2 \gamma - 2 c \beta \gamma^2 \\ 
&\phantom{G_1^4(y_1,y_2)=}+2 b y_1 \beta \gamma^2 + 2 b y_2 \beta \gamma^2 + 2 b \gamma^3, \\
&F_2^4(y_1,y_2)= G_2^4(y_1,y_2)=0.
\end{aligned}
\end{equation}
Eventually the system \eqref{sistemageral4} could have a continuum of solutions $(y_1,y_2)$, but then the possible periodic solutions would not be limit cycles. Therefore we assume that this system  has finitely many solutions. 

Assume that $q_{41}(y_1) \neq 0$. From equation $F_1^4(y_1,y_2)=0$ we get
\begin{equation}\label{valory2teo4}
y_2= \dfrac{p_{41}(y_1)}{q_{41}(y_1)},
\end{equation}
where 
\[ 
\begin{aligned}
& p_{41}(y_1)= -2 b c + 2 c^3 \beta - 2 b c^2 \gamma - 2 \beta \gamma - 2 c \beta \gamma^2 + 2 b \gamma^3 + \\
& \phantom{p_{41}(y_1)=}+y_1 (-b^2 + 2 b c^2 \beta - \beta^2 - 2 b^2 c \gamma - 2 c \beta^2 \gamma + 2 b \beta \gamma^2), \\
& q_{41}(y_1)= \; b^2 - 2 b c^2 \beta + \beta^2 + 2 b^2 c \gamma + 2 c \beta^2 \gamma - 2 b \beta \gamma^2 + \\
& \phantom{q_{41}(y_1)=} + y (-2 b^2 c \beta + 2 c \beta^3 + 2 b^3 \gamma - 2 b \beta^2 \gamma).
\end{aligned} 
\]
We substitute \eqref{valory2teo4} in $F_2^4(y_1,y_2)=0$, and we get a rational function. The polynomial $p_4(y_1)$ of the numerator has degree $6$. Therefore  $p_4(y_1)$ has at most six roots, and consequently system  \eqref{sistemageral4} has at most six solutions. Therefore the discontinuous piecewise differential system has at most three limit cycles.

Assume that $q_{41}(y_1)= 0$. Then $p_{41}(y_1)= 0$, and there is at most one solution for $y_1$, and consequently the discontinuous piecewise differential system has at most one limit cycle in this case.
	
In summary, we have proved these discontinuous piecewise differential systems have at most three limit cycles. Now we shall prove that the differential system formed by the differential systems (\ref{sistema4positivo}) and (\ref{sistema4negativo}) has one limit cycle. 

In $x\ge 0$ the differential system (\ref{sistema4positivo}) comes from the cubic isochronous center $(S_1^3)$ after doing the affine change of variables $(x,y)\to (1 + x,-2 + y)$, and in $x\le 0$ the differential system (\ref{sistema4negativo})  comes from the cubic isochronous center $(S_2^3)$ after reversing the independent variable of sign. The first integrals of these two differential system are
\begin{equation*}
H_{1c}(x,y) = \dfrac{(1 + x)^2 + (-2 + y)^2}{1 + 2 (1 + x) (-2 + y)}, \quad \mbox{and} \quad H_{2c}(x,y)= \dfrac{(x^2 + y^2)^2}{1 + 4xy},  
\end{equation*}
respectively.  Therefore system \eqref{x3} becomes
\[ 
\begin{aligned}
& 2 - 3 y_1 - 3 y_2 + 2 y_1 y_2=0,\\
& (y_1 + y_2) (y_1^2 + y_2^2) =0.
\end{aligned} 
\]
Its solution is $(y_1^{4*},y_2^{4*})= \left(-1, 1 \right)$, note that $y_1^{4*}<y_2^{4*}$.  

The solution $(x_1^4(t), y_1^4(t))$ in $x\ge 0$ of system \eqref{sistema4positivo} such that $(x_1^4(0),y_1^4(0)) = (0, y_1^{4*})$ is contained in the level curve $H_{2c}(0,y_2^{4*})= 1 = H_{2c}(0,y_1^{4*})$, i.e. in the curve
\[ 
\begin{aligned}
H_{2c}(x,y) = \dfrac{(x^2 + y^2)^2}{1 + 4xy} = 1.
\end{aligned} 
\]
The solution $(x_2^4(t), y_2^4(t))$ in $x\le 0$ of system \eqref{sistema4negativo} with $(x_2^4(0),y_2^4(0)) = (0,y_2^{4*})$ is contained in level curve $H_{1c}(0,y_2^{4*})= -2 = H_{1c}(0,y_1^{4*})$, i.e. in the curve
\[ 
\begin{aligned}
H_{1c}(x,y)= \dfrac{(1 + x)^2 + (-2 + y)^2}{1 + 2 (1 + x) (-2 + y)} = - 2.
\end{aligned} \]
that is,

Drawing the orbits $(x_k^4(t), y_k^4 (t))$, $k=1,2$, we obtain the limit cycle of Figure \ref{cicloteo4}.  
\end{proof}

The problem if the upper bound of Proposition \ref{teo4} for the maximum number of limit cycles of that class of discontinuous piecewise differential systems is reached or not remains open.

\begin{figure}[!htb]
\begin{center}
\includegraphics[scale=0.4]{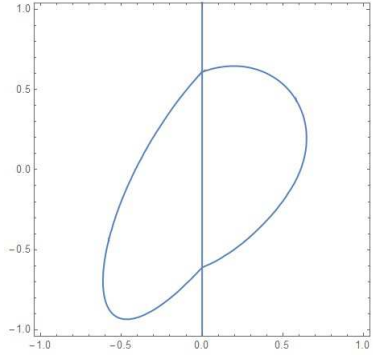}
\caption{The limit cycle of the discontinuous piecewise differential system \eqref{sistema5positivo}-\eqref{sistema5negativo} of Proposition \ref{teo5}.}\label{cicloteo5}
\end{center}
\end{figure}

The proof that $N_{clc}[(S_2^3),(S_2^3)]\le 9$ is given in the next proposition.

\begin{prop}\label{teo5}
Consider the class of discontinuous piecewise differential systems separated by the straight line $x=0$ and formed by two distinct cubic isochronous center $(S_2^3)$ after an arbitrary affine change of variables. Then these differential systems can have at most nine limit cycles.  Moreover, the discontinuous piecewise differential system in this class formed by the differential system		
\begin{equation}\label{sistema5positivo}
\dot{x}= -y + x^3 - 3 x y^2, \qquad \dot{y}=x + 3 x^2 y - y^3,
\end{equation}
in $x\ge 0$, and the differential system
\begin{equation}\label{sistema5negativo}
\begin{array}{l}
\dot{x}= -1 - 2 x^3 - 5 y - 3y^2 + x^2 (6 + 6 y) + x (5 - 3 y^2), \\
\dot{y}=  -4 x^3 - 5 y + 6 x^2 y - 6 y^2 - y^3 + x (10 + 12 y),
\end{array}
\end{equation}
in $x\le 0$ has one limit cycle.
\end{prop}

\begin{proof}
Under the assumptions of the proposition if such discontinuous piecewise differential systems have a limit cycle intersecting the discontinuity straight line $x=0$ in the two points $(0,y_1)$ and $(0,y_2)$, then $y_1$ and $y_2$ must satisfy that 
\begin{equation}\label{x4}
H_{2c}(0,y_1)- H_{2c}(0,y_2)= 0, \quad \mbox{and} \quad H_{2_1c} (0, y_1) - H_{2_1c} (0, y_2) = 0,
\end{equation}
where
\begin{equation*}
H_{2c}(x,y)= \dfrac{((c + a x + b y)^2 + (x \alpha + y \beta + \gamma)^2)^2}{1 + 4 (c + a xX + b y) (x \alpha + y \beta + \gamma)},
\end{equation*}
and
\begin{equation*}
H_{2_1c}(x,y)= \dfrac{((c_1 + a_1 x + b_1 y)^2 + (x \alpha_1 + y \beta_1 + \gamma_1)^2)^2}{1 + 4 (c_1 + a_1 x + b_1 y) (x \alpha_1 + y \beta_1 + \gamma_1)}.
\end{equation*}
Then system \eqref{x4} becomes
\begin{equation*}
\begin{aligned}
&G_1^5(y_1,y_2)=-(((y_1 - y_2)A_1^5(y_1,y_2)))/\\
&\phantom{G_2^4(y_1,y_2)=} \; ((1 + 4 c y_1 \beta + 4 b y_1^2 \beta + 4 c \gamma + 4 b y_1 \gamma) \\
& \phantom{G_2^4(y_1,y_2)=} \; (1 + 4 c y_2 \beta + 4 b y_2^2 \beta + 4 c \gamma + 4 b y_2 \gamma)))=0, \\
&G_2^5(y_1,y_2)=-(((y_1 - y_2)A_2^5(y_1,y_2)))/\\
&\phantom{G_2^4(y_1,y_2)=} \; ((1 + 4 c_1 y_1 \beta_1 + 4 b_1 y_1^2 \beta_1 + 4 c_1 \gamma_1 + 4 b_1 y_1 \gamma_1) \\
& \phantom{G_2^4(y_1,y_2)=} \; (1 + 4 c_1 y_2 \beta_1 + 4 b_1 y_2^2 \beta_1 + 4 c_1 \gamma_1 + 4 b_1 y_2 \gamma_1)))=0,
\end{aligned}
\end{equation*}
where $A_{i}^5(y_1,y_2)$ for $i=1,2$ are polynomials of degree $5$. Since the denominators of the previous system cannot be zero and $y_1<y_2$, this system is equivalent to the system  
\begin{equation*}\label{sistemageral5}
A_1^5(y_1,y_2)=0, \qquad A_2^5(y_1,y_2)= 0.
\end{equation*}
where
\begin{equation}\label{sistemageral51}
\begin{aligned}
&A_1^5(y_1,y_2)= A_0 +A_1(y_1+y_2)+A_2(y_1^2+y_2^2)+A_3y_1y_2+ A_4(y_1^3+y_2^3) \\
&\phantom{F_2^5(y_1,y_2)=}+A_5(y_1^2y_2+y_1y_2^2) + A_6(y_1^3y_2+y_1y_2^3)+A_7 y_1^2 y_2^2 +A_8(y_1^3y_2^2+y_1^2y_2^3), \\
&A_2^5(y_1,y_2)= B_0 +B_1(y_1+y_2)+B_2(y_1^2+y_2^2)+B_3y_1y_2+ B_4(y_1^3+y_2^3) \\
&\phantom{F_2^5(y_1,y_2)=}+B_5(y_1^2y_2+y_1y_2^2) + B_6(y_1^3y_2+y_1y_2^3)+B_7 y_1^2 y_2^2 +B_8(y_1^3y_2^2+y_1^2y_2^3).
\end{aligned}
\end{equation}
The above polynomials can be rewritten in the variables $z$ and $w$, where $z = y_1 + y_2$ and $w = y_1y_2$. For example, $ y_1 ^ 2 + y_2 ^ 2 = (y_1 + y_2) ^ 2-2 y_1 y_2 = z ^ 2-2w $. In these new variables system \eqref{sistemageral51} becomes
\begin{equation}\label{sistemageral52}
\begin{aligned}
&P_1^5(z,w)= A_0 +A_1z+(-2A_2+A_3)w+A_4z^3+(A_5-3A_4)zw+A_2z^2+A_6z^2w+\\
&\phantom{F_2^5(y_1,y_2)=} +(-2A_6+A_7)w^2+A_8zw^2 =0, \\
&P_1^5(z,w)= B_0 +B_1z+(-2B_2+B_3)w+B_4z^3+(B_5-3B_4)zw+B_2z^2+B_6z^2w+\\
&\phantom{F_2^5(y_1,y_2)=} +(-2B_6+B_7)w^2+B_8zw^2 =0.
\end{aligned}
\end{equation}
	
Eventually system \eqref{sistemageral52} could have a continuum of solutions $(z,w)$, but then the possible periodic solutions would not be limit cycles. So we assume that this system has a finite number of solutions. The two equations in system \eqref{sistemageral52} are polynomials of degree $3$, then by Bezout Theorem the discontinuous piecewise differential system has a maximum of $9$ limit cycles.

Now we will prove that the discontinuous piecewise differential system formed by the differential systems \eqref{sistema5positivo} and \eqref{sistema5negativo} has one limit cycle.

In $x\ge 0$ the differential system (\ref{sistema5positivo}) is the cubic isochronous center $(S_2^3)$, and in $x\le 0$ the differential system (\ref{sistema5negativo}) comes from the cubic isochronous center $(S_2^3)$ after the affine change of variables $(x,y)\to (1 + x,-1 + x - y)$. The first integrals of these two differential system are
\begin{equation*}
H_{2c}(x,y)= \dfrac{(x^2 + y^2)^2}{1 + 4xy}, \quad  \mbox{and} \quad  H_{2_1c}(x,y) = \dfrac{((1 + x)^2 + (-1 + x - y)^2)^2}{1 + 4 (1 + x) (-1 + x - y)},
\end{equation*}
respectively. So system \eqref{x4} becomes
\[ 
\begin{aligned}
& (y_1 + y_2) (y_1^2 + y_2^2)=0,\\
& -8 - 24 y_1 - 12 y_1^2 - 3 y_1^3 - 24 y_2 - 44 y_1 y_2 - 19 y_1^2 y_2 - 4 y_1^3 y_2 - 12 y_2^2 - 19 y_1 y_2^2 -  \\
&-4 y_1^2 y_2^2-3 y_2^3 - 4 y_1 y_2^3 =0.
\end{aligned} 
\]
Its solution is $(y_1^{5*},y_2^{5*})= \left(-\sqrt{(-5+\sqrt{33})/2}, \sqrt{(-5+\sqrt{33})/2} \right)$, note that $y_1^{5*}<y_2^{5*}$. 

The solution $(x_1^5(t), y_1^5(t))$ in $x\ge 0$ of the system \eqref{sistema5positivo} such that $(x_1^5(0),y_1^5(0)) = (0, y_1^{5*})$ is contained in the level curve  $H_{2c}(0,y_2^{5*})= (-5+\sqrt{33})^2/4 = H_{2c}(0,y_1^{5*})$, i.e. in the curve
\[ 
\begin{aligned}
H_{2c}(x,y) = \dfrac{(x^2 + y^2)^2}{1 + 4xy} = \dfrac{(-5+\sqrt{33})^2}{4}.
\end{aligned} 
\]
The solution $(x_2^5(t), y_2^5(t))$ in $x\le 0$ of the system \eqref{sistema5negativo} with $(x_2^5(0),y_2^5(0)) = (0,y_2^{5*})$ is contained in the level curve $H_{2_1c}(0,y_2^{5*})= (1-\sqrt{33})/2 = H_{2_1c}(0,y_1^{5*})$, i. e. in the curve
\[ 
\begin{aligned}
H_{2_1c}(x,y)= \dfrac{((1 + x)^2 + (-1 + x - y)^2)^2}{1 + 4 (1 + x) (-1 + x - y)}= \dfrac{(1-\sqrt{33})}{2}.
\end{aligned} 
\]

Drawing the orbits $(x_k^5(t), y_k^5 (t))$, $k=1,2$, we obtain the limit cycle of Figure \ref{cicloteo5}.
\end{proof}

The problem if the upper bound of Proposition \ref{teo5} for the maximum number of limit cycles of that class of discontinuous piecewise differential systems is reached or not remains open.

\section*{Acknowlegements}

This work has received funding from the Agencia Estatal de Investigaci\'on grant
PID2019-104658GB-I00, the H2020 European Research Council grant
MSCA-RISE-2017-777911, CNPq 304798/2019-3 grant, and São Paulo Paulo Research Foundation (FAPESP) grants 2021/14695-7, 2019/10269-3, 2018/05098-2, and 2016/00242-2.



%

\end{document}